\documentclass[11pt]{amsart}

\usepackage{latexsym}
\usepackage{amsfonts}

\newcommand{\field}[1]{\mathbb{#1}}
\newcommand{\C}{\field{C}}

\newtheorem{defi}{Definition}[section]
\newtheorem{ex}[defi]{Example}
\newtheorem{lem}[defi]{Lemma}
\newtheorem{theo}[defi]{Theorem}

\newtheorem{pr}[defi]{Proposition}

\font\tenmsy=msbm10

\def\Bbb#1{\hbox{\tenmsy#1}} 

\title[On the set of fixed points of a polynomial automorphism]{On the set of fixed points of a polynomial automorphism}
\makeatletter

\@addtoreset{equation}{section}
\makeatother
\author{Zbigniew Jelonek \& Tomasz Lenarcik}
\address[Z. Jelonek]{Instytut Matematyczny\\
Polska Akademia Nauk\\
\'Sniadeckich 8, 00-956 Warszawa, Poland}
\email{najelone@cyf-kr.edu.pl}
\address[T. Lenarcik]{Faculty of
Mathematics and Computer Science, Jagiellonian Univeristy\\
ul. prof. Stanis\l{}awa \L{}ojasiewicza 6, 30-348 Krak\'ow,
Poland\\} \email{Tomasz.Lenarcik@im.uj.edu.pl}

 \subjclass{14
R 10}
\thanks{The author was partially supported by the grant of NCN,  2014-2017. }

\date{\today}

\begin{document}

\maketitle

\begin{abstract}
Let $\Bbb K$ be an algebraically closed field of characteristic
zero. We say that a polynomial automorphism $f: \Bbb K^n\to \Bbb
K^n$ is special if the Jacobian of $f$ is equal to $1.$ We show
that every $(n-1)$-dimensional component $H$ of the set ${\rm
Fix}(f)$ of fixed points of a non-trivial special polynomial
automorphism $f: \Bbb K^n\to\Bbb K^n$ is uniruled. Moreover, we
show that if $f$ is non-special and $H$ is an $(n-1)$-dimensional
component of the set ${\rm Fix}(f)$, then $H$ is smooth,
irreducible and $H={\rm Fix}(f)$ and for $\Bbb K=\C$ the Euler
characteristic of $H$ is equal to $1.$
\end{abstract}

\bibliographystyle{alpha}
\maketitle

\section{Introduction}

Polynomial automorphism of affine space $\Bbb K^n$ have always
attracted a lot of attention, but the nature of these
automorphisms still not well-known.

Here we are interested in the set of fixed points of  such
automorphisms. Let us recall that if $f: \Bbb K^2\to\Bbb K^2$ is a
polynomial automorphism, then the set ${\rm Fix}(f)$ of fixed
points of $f$ is either finite, or it is a union of smooth,
disjoint curves which all are isomorphic to $\Bbb K.$ This result
was proved in \cite{jel1} and later it was partially reproved in
\cite{miya}. Moreover, by Kambayashi result, every automorphism of
$\Bbb K^2$ of finite order is linear in some system of
coordinates. We do not know whether Kambayashi result can be
extended to higher dimensions. However there is some evidence that
the set of fixed points of a polynomial automorphism of finite
order should be isomorphic to a linear subspace.

In higher dimensions the situation is more complicated. The set of
fixed points can have components of dimension $n-1$ and
additionally less dimensional components - an easy example is $f :
(x,y,z)\ni \Bbb K^4\to (x+zy, y+ zw, z, w)\in \Bbb K^4.$ Moreover,
an $(n-1)$-dimensional component of the set of fixed points of $f$
can be a singular variety- as in the  famous Nagata automorphism:
$$f(x,y,z)=(x+(x^2-yz)z, y+2(x^2-yz)x+(x^2-yz)^2z, z).$$
Here the set of fixed points  is the quadratic cone $\Lambda=\{
(x,y,z): x^2=yz\}$. We show however that such a strange behavior
is possible only for {\it special automorphisms, i.e., for
automorphisms with Jacobian equal to one}. In this paper we focus
on $(n-1)$-dimensional components of the set of fixed points of
polynomial automorphism of $\Bbb K^n.$ Our first result is:

\begin{theo}
Let $\Bbb K$ be an arbitrary algebraically closed field. Let
$f:\Bbb K^n\to \Bbb K^n$ be a non-special polynomial automorphism.
Let $H$ be a hypersurface that is contained in the set ${\rm
Fix}(f)$ of fixed points of $f.$ Then

1) $H$ is smooth and irreducible.

2) $H={\rm Fix}(f).$
\end{theo}

We can say  more if the order of $f$ is infinite and the field
$\Bbb K$ has characteristic zero. Let $H\subset \Bbb K^n$ be a
hypersurface and $I(H)=(p).$ We say that $H$ is {\it super-smooth}
if $(\frac{\partial p}{\partial x_1}(x),..., \frac{\partial
p}{\partial x_n}(x))\not=0$  for every  $x\in \Bbb K^n$. We have:

\begin{theo}
Let $\Bbb K$ be an algebraically closed field of characteristic
zero. Let $f:\Bbb K^n\to \Bbb K^n$ be a non-special polynomial
automorphism of infinite order. Let $H$ be a hypersurface that is
contained in the set ${\rm Fix}(f)$ of fixed points of $f.$  Then

1) $H$ is smooth and irreducible.

2) $H={\rm Fix}(f).$

3) If $\lambda={\rm Jac}(f)$ has finite order in $\Bbb K^*$, then
$H$ is uniruled.

4) If $\lambda={\rm Jac}(f)$ has infinite order in $\Bbb K^*$,
then $H$ is super-smooth.  Moreover, if $\Bbb K=\C$, then the
Euler characteristic of $H$ is equal to $1.$
\end{theo}

\noindent For special automorphisms of infinite order, the set of
fixed points can have many $(n-1)$-dimensional components; the
easiest example is a triangular automorphism
$$f(x,y,z)=(x+ \prod^r_{i=1} h_i(y,z), y, z).$$ Moreover, as we
noticed before, such an $(n-1)$-dimensional component can be
singular. Additionally, if $H$ is the union of all
$(n-1)$-dimensional components of ${\rm Fix}(f)$, then in general
$H\not={\rm Fix}(f).$ However, the $(n-1)$-dimensional components
of the set of fixed points of a special automorphism have one
common property - they are uniruled:

\begin{theo}
Let $\Bbb K$ be an algebraically closed field of characteristic
zero. Let $f:\Bbb K^n\to \Bbb K^n$ be a non-trivial special
polynomial automorphism. Let $H$ be a hypersurface that is
contained in the set ${\rm Fix}(f)$ of fixed points of $f.$ Then
$H$ is uniruled, i.e., it is covered by rational curves.
\end{theo}

\section{Non-special automorphisms}
We  first need an elementary lemma from linear algebra.

\begin{lem}\label{linear}
Let $X=\Bbb K^n$ and  let $F: X\to X$ be a linear isomorphism.
Assume that there exists a hyperplane $W$ which is contained in
the set of fixed points of $F.$ Then all eigenvalues of $F$ are
$1$ (of multiplicity at least $n-1$)  and ${\rm det}(F).$

Now assume that ${\rm det}(F)=\lambda\not=1.$ If $l$ is a linear
form such that ${\rm ker}\ l=W$, then the forms proportional to
the form $l$ are the only eigenvectors of $F^*$ with eigenvalue
$\lambda.$
\end{lem}

\begin{proof}
Since $W$ is contained in the set of fixed points of $F$, we have
that $1$ is an eigenvalue of $F$  of multiplicity at least dim
$W=n-1.$ Hence the remaining eigenvalue has to be equal to ${\rm
det}(F)=\lambda.$

Now assume that $\lambda\not=1$ and let $g$ be an eigenvector with
eigenvalue $\lambda.$ Let $\{w_1,..., w_{n-1}\}$ be a basis of
$W.$ It is easy to see that $\{w_1,..., w_{n-1}\}$ and $g$ form a
basis of $X.$ In this basis $F$ is given by the formula
$$F(x_1,...,x_{n-1}, x_n)=(x_1,..., x_{n-1}, \lambda x_n).$$
In particular, the hyperplane $W$ is described by the form $x_n$,
i.e., $l=cx_n$ for some $c\in \Bbb K^*.$ Moreover,
$$F^*(\sum^n_{i=1} a_i x_i)=a_1x_1+...+a_{n-1}x_{n-1}+ \lambda a_n
x_n.$$ Consequently, only the forms proportional to $x_n$ have
eigenvalues equal to $\lambda.$
\end{proof}

Moreover, we have:

\begin{lem}
Let $f:\Bbb K^n\to \Bbb K^n$ be a polynomial automorphism. Let
$H\subset \Bbb K^n$ be a hypersurface such that $H\subset \Bbb
K^n.$ If $H$ is singular, then ${\rm Jac}(f)=1.$
\end{lem}

\begin{proof}
Let $a\in {\rm Sing}(H).$ Choose a system of coordinates in which
$a=(0,..., 0).$ Let $h$ be a reduced equation of $H.$ Since
$H\subset {\rm Fix}(f)$ we have that $f_i-x_i$ vanishes on $H$,
i.e., $h| f_i-x_i.$ Consequently,
$$f_i=x_i+a_ih, \ i=1,...,n.$$ Since $h=\sum_{|\alpha|\ge 2} h_\alpha x^\alpha,$
we have ${\rm Jac}(f)={\rm Jac}({\rm identity})=1.$
\end{proof}

\begin{lem}\label{jac}
Let $f:\Bbb K^n\to \Bbb K^n$ be a polynomial automorphism. Let
$H\subset \Bbb K^n$ be an irreducible hypersurface, such that
$H\subset {\rm Fix}(f).$ Let $h$ be a reduced equation of $H.$
Then $h\circ f=\lambda h$, where $\lambda={\rm Jac}(f).$
\end{lem}

\begin{proof}
Let $a\in H$ be a smooth point of $H.$ Take $W=T_a H$, $X=T_a \Bbb
K^n$ and $F=d_a f.$ By the assumption, the subspace $W$ is
contained in the set of fixed points of the linear isomorphism
$F.$ Moreover, $W$ is described by the linear form $l=\sum^n_{i=1}
\frac{\partial h}{\partial x_i}(a) x_i=0.$ This form can be
identified with the vector ${\rm grad}_a h= (\frac{\partial
h}{\partial x_1}(a),..., \frac{\partial h}{\partial x_n}(a)).$

Since $f$ is a polynomial automorphism and $H\subset {\rm
Fix}(f),$ the polynomial $h$ describes the same hypersurface as
the polynomial $h(f).$ Note that these two polynomials are reduced
and consequently they are generators of the same ideal $I(H).$
This means that there exists a constant $c\in \Bbb K^*$ such that
$h(f)=c\cdot h.$ After differentiation, we have $(d_a f)^* {\rm
grad}_a h= c\cdot {\rm grad}_a h.$ Hence the vector ${\rm grad}_a
h$ is an eigenvector of $F^*.$ Now if $\lambda={\rm Jac}(f)=1$,
then by Lemma \ref{linear}, we have that all eigenvalues of $F$
(and hence also of $F^*$) are equal to $1.$ Consequently,
$c=\lambda=1.$ If $\lambda\not=1$, then again by Lemma
\ref{linear}, we have $c=\lambda.$
\end{proof}

Now we are ready to prove:

\begin{theo}\label{th1}
Let $f:\Bbb K^n\to \Bbb K^n$ be a non-special polynomial
automorphism. Let $H$ be a hypersurface that is contained in the
set ${\rm Fix}(f)$ of fixed points of $f.$ Then

1) $H$ is smooth and irreducible.

2) $H={\rm Fix}(f).$
\end{theo}

\begin{proof}
Let $S$ be an irreducible component of $H$ with reduced equation
$s=0.$ By Lemma \ref{jac}, we have $s\circ f=\lambda s$ and
$\lambda\not=0.$ For $t\not=0,$ the hypersurface $S_t:=\{ x:
s(x)=t\}$ is transformed by $f$ onto the hypersurface $S_{\lambda
t}.$ Since $S_t\cap S_{\lambda t}=\emptyset$ for $t\not=0,$ we
have ${\rm Fix}(f)=S_0=S.$ In particular, $H=S$ is an irreducible
hypersurface.
\end{proof}

Now we show that non-trivial automorphisms of finite order and
with large  set of fixed points  cannot be special. We start with:

\begin{lem}\label{lemat}
Let $L:\C^n\to \C^n$ be a linear mapping of finite order $m>1$ and
assume that the set of fixed points of $L$ is a hyperplane $W$.
Then in some coordinates, $$L(x_1, x_2, \ldots, x_n)=(\epsilon
x_1, x_2,\ldots, x_n),$$ where $\epsilon^m=1$ and
$\epsilon\not=1.$
\end{lem}

\begin{proof}
Take a basis $e_1,\ldots, e_n$ in $\C^n$ such that $e_2,\ldots,
e_{n}$ span the hyperplane $W.$ Hence $L(e_i)=e_i$ for $i>1$ and
$L(e_1)=\sum^n_{i=1} a_i e_i.$ In particular, $L(x_1
e_1+\ldots+x_n e_n)=(a_1x_1)e_1+(x_2+a_2x_1)e_2+
\ldots+(x_{n}+a_{n}x_1)e_{n}.$ Since $L$ has  finite order, we
have $a_i=0$ for $i>1.$ In particular, $L(x_1,\ldots, x_n)=(a_1
x_1, x_2,\ldots, x_{n}).$ However {\rm det} $L^m= 1$, i.e.,
$a_1=\epsilon,$ where $\epsilon^m=1$ and $\epsilon\not=1.$
\end{proof}

Now we can state:

\begin{pr}\label{special}
Let $\Phi: \Bbb C^n\to\Bbb C^n$ be a polynomial automorphism of
finite order $m>1.$ Assume that the set  of fixed points of $\Phi$
has dimension $n-1.$ Then $\Phi$ is not special.
\end{pr}

\begin{proof}
Let $H\subset {\rm Fix}(\Phi)$ be a hypersurface and let $x\in H.$
By the Cartan Theorem (see \cite{car}), the mapping $\Phi$ is
holomorphically linearizable in some neighborhood of $x.$ Now the
proof reduces to Lemma \ref{lemat}.
\end{proof}

We conclude this section by:

\begin{theo}\label{th2}
Let $\Bbb K$ be an algebraically closed field of characteristic
zero. Let $f:\Bbb K^n\to \Bbb K^n$ be a non-special polynomial
automorphism of infinite order. Let $H$ be a hypersurface that is
contained in the set ${\rm Fix}(f)$ of fixed points of $f.$ Then

1) $H$ is smooth and irreducible.

2) $H={\rm Fix}(f).$

3) If $\lambda={\rm Jac}(f)$ has finite order in $\Bbb K^*$, then
$H$ is uniruled.

4) If $\lambda={\rm Jac}(f)$ has infinite order in $\Bbb K^*$,
then $H$ is super-smooth.  Moreover, if $\Bbb K=\C$, then the
Euler characteristic of $H$ is equal to $1.$
\end{theo}

\begin{proof}
By the Lefschetz principle, we can assume that $\Bbb K=\C.$ We
have 1) and 2) by Theorem \ref{th1}. The point 3) follows from
Theorem \ref{th3} below. Hence it is enough to prove 4). Let
$H={\rm Fix}(f).$ As we know the hypersurface $H$ is irreducible.
Let $h=0$ be an irreducible equation for $H.$ By Lemma \ref{jac}
we have $h\circ f=\lambda h.$

We show that  the polynomial $h$ has no atypical values, except
possibly $0.$  Indeed, since the fiber $h=t$ is transformed by $f$
onto the fiber $h=\lambda t$ and $\lambda$ has an infinite order
we have that for $t\not=0$ the fiber $h=t$ cannot by atypical
(there is only a finite number of atypical values of $h$-for
details see e.g., \cite{j-k}). In particular, all fibers $h=t$ for
$t\not=0$ are homeomorphic and $f: \C^n\setminus H\to \C^*$ is a
locally trivial fibration. Computing the  Euler characteristics,
we have $\chi(H_1)\chi(\C^*)+\chi(H)=\chi(\C^n)=1$, i.e.,
$\chi(H)=1.$

Moreover, all fibers of $h$ are smooth. For $H_0$ it follows from
1), if $t\not=0$, then $H_t$ is a typical fiber, hence it is
smooth.
\end{proof}

\section{Main Result}

First we recall the following important fact (see \cite{jel3}):

\begin{theo}\label{mathann}
Let $\Bbb K$ be an algebraically closed field of characteristic
zero. Let $X$ be a quasi-affine variety.  If the group ${\rm
Aut}(X)$ is infinite, then $X$ is uniruled.
\end{theo}

Now we can prove our main result:

\begin{theo}\label{th3}
Let $\Bbb K$ be an algebraically closed field of characteristic
zero. Let $f:\Bbb K^n\to \Bbb K^n$ be a non-trivial special
polynomial automorphism. Let $H$ be a hypersurface that is
contained in the set ${\rm Fix}(f)$ of fixed points of $f.$ Then
$H$ is uniruled, i.e., it is covered by rational curves.
\end{theo}

\begin{proof}
By Proposition \ref{special}, automorphism $f$ has an infinite
order. Let $h$ be an irreducible equation of $H.$ By Lemma
\ref{jac} we have $h\circ f=h.$ In particular, $f$ preserves all
fibers $h=t.$ Let $\Gamma$ be a non-rational affine curve and let
$\pi: \Gamma \to \C$ be a finite morphism. Denote by
$$X:=\C^n\times_\C \Gamma,$$
the fiber product determined by the mappings $h$ and $\pi.$ On $X$
acts the automorphism $F=(f\times {\rm identity})_{|_X}.$ We have
the projection $\Pi: X\to \Gamma$ with $(n-1)$-dimensional fibers
$\Pi^{-1}(\gamma)= h^{-1}(\pi(\gamma)).$ Because the polynomial
$h$ is irreducible,  generic fibers of $\Pi$ are irreducible.
Moreover, since the variety $X$ is a hypersurface in $\C^n\times
\Gamma$, it has to be   an irreducible affine variety of dimension
$n$. It is easy to see that the automorphism $F$ has an infinite
order. By virtue of Theorem \ref{mathann}, the variety $X$ is
uniruled. Since the curve $\Gamma$ is not uniruled, all fibers of
the mapping $\Pi: X\to \Gamma$ are uniruled. The hypersurface $H$
is one of these fibers.
\end{proof}

\begin{ex}
{\rm  Let $$N: \C^3\ni (x,y,z) \mapsto (x - 2y (xz + y^2) - z (xz
+ y^2)^2, y + z(xz + y^2), z)\in \C^3,$$ be the famous Nagata
automorphism.   The set of fixed points of $N$ is the cone
$\Lambda=\{ (x,y,z)\in\C^3 : xz + y^2=0\}.$ Since $\Lambda$ is a
singular variety, we see  that the automorphism $N$  is special.
Of course $\Lambda$ is a uniruled hypersurface.}$\square$
\end{ex}

\begin{ex}
{\rm We cannot expect that a lower dimensional components of the
set ${\rm Fix}(f)$ are uniruled. Indeed, let $\Gamma=\{ (x,y)\in
\C^2 : h(x,y)=0\}$ be an arbitrary plane curve. Then there is a
polynomial automorphism $f:\C^3\to \C^3$ such that ${\rm
Fix}(f)\cong\Gamma.$ Indeed, take $$f(x,y,z)=(x, y+z+h(x,y),
z+h(x,y)).$$ If we take $g(x,y,z,w)=(f(x,y,z), 2w)$, we obtain a
non-special automorphism $g:\C^4\to\C^4$ with ${\rm Fix}(g)\cong
\Gamma.$}
\end{ex}

At the end of this paper we state a Conjecture, which is more or
less the Masuda-Miyanishi Conjecture (see \cite{miya}):

\vspace{3mm}

{\bf Conjecture.} {\it Let $F:\C^n\to\C^n$ be a non-special
automorphism. Assume that $H\subset {\rm Fix}(f)$ is a
hypersurface. Then $H$ is isomorphic to $\C^{n-1}.$}

\end{document}